\documentclass[12pt,reqno]{amsart}
\usepackage{amsmath,amsthm}
\usepackage{caption}
\usepackage{mathtools}
\usepackage{cite}
\usepackage{enumitem}
\usepackage{color}
\usepackage{url}
\usepackage[usenames,dvipsnames]{xcolor}
\usepackage{graphicx}
\usepackage{xspace}
\usepackage[unicode=true,colorlinks]{hyperref}
\hypersetup{
linkcolor=blue,
citecolor=blue,
}
\usepackage{cleveref}
\usepackage{leftidx}
\usepackage[margin=1.5in]{geometry}
\usepackage{comment}

\usepackage{mathrsfs}

\newcommand{\SL}{\operatorname{SL}}

\newcommand{\Gr}{\operatorname{Gr}}
\usepackage{amssymb}

\makeatletter
\renewcommand{\paragraph}{%
	\@startsection {paragraph}{4}
	{\z@} \z@ {-\fontdimen 2\font }\bfseries
}
\makeatother

\makeatletter 
\def\@cite#1#2{{\m@th\upshape\bfseries%
		[{#1\if@tempswa{\m@th\upshape\mdseries, #2}\fi}]}}
\makeatother 

\numberwithin{equation}{section}

\theoremstyle{plain}
\newtheorem{thm}{Theorem}[section]
\newtheorem{cor}[thm]{Corollary}
\newtheorem{prop}[thm]{Proposition}
\newtheorem{lem}[thm]{Lemma}

\theoremstyle{definition}

\theoremstyle{remark}
\newtheorem{remark}[thm]{Remark}

\newcommand{\ann}[1]{}

\newcommand\bfp{\mathbf{p}}
\newcommand\bfq{\mathbf{q}}
\newcommand\bfs{\mathbf{s}}

\newcommand\bfx{\mathbf{x}}

\newcommand\bbP{\mathbb{P}}

\newcommand\bbR{\mathbb{R}}

\newcommand\bbZ{\mathbb{Z}}

\newcommand\cI{\mathcal{I}}
\newcommand{\cJ}{\mathcal{J}} 

\newcommand\cL{\mathcal{L}}

\newcommand{\R}{\mathbb{R}}

\newcommand{\Q}{\mathbb{Q}}
\newcommand{\Z}{\mathbb{Z}}
\newcommand{\N}{\mathbb{N}}

\DeclarePairedDelimiter\abs{\lvert}{\rvert}%
\DeclarePairedDelimiter\norm{\lVert}{\rVert}%

\makeatletter
\let\oldabs\abs
\def\abs{\@ifstar{\oldabs}{\oldabs*}}
\let\oldnorm\norm
\def\norm{\@ifstar{\oldnorm}{\oldnorm*}}
\makeatother

\newcommand{\eps}{\varepsilon}
\newcommand{\de}{\mathrm{d}}

\newcommand{\at}{\widetilde{\alpha}}

\newcommand{\pifix}{\pi_{\mathrm{fix}}}

\DeclareMathOperator{\diag}{diag}

\DeclareMathOperator{\Ima}{image}
\DeclareMathOperator{\Mat}{Mat}

\DeclareMathOperator{\Ad}{Ad}

\title[Singular vectors on affine subspaces]{An upper bound of the Hausdorff dimension of singular vectors on affine subspaces}
\author[Shah]{Nimish~A.~Shah}
\address{The Ohio State University, Columbus, OH 43210}
\email{shah@math.osu.edu}
\author[Yang]{Pengyu~Yang}
\address{Morningside Center of Mathematics, Chinese Academy of Sciences, Beijing 100190}
\email{yangpengyu@amss.ac.cn}

\thanks{Pengyu Yang is supported by National Key R\&D Program of China 2022YFA1007500 and NSFC grant 22AAA00245.}

\begin{document}
	
	\begin{abstract}
		In Diophantine approximation, the notion of singular vectors was introduced by Khintchine in the 1920's. We study the set of singular vectors on an affine subspace of $\R^n$. We give an upper bound of its Hausdorff dimension in terms of the Diophantine exponent of the parameter of the affine subspace.
	\end{abstract}
	
	\subjclass[2010]{Primary 37A17, 11J83; Secondary 22E46, 14L24, 11J13}
	\keywords{Homogeneous dynamics, Diophantine approximation}
	\maketitle

	
	

	\section{Introduction}
	Let $n\geq 2$ be a positive integer. A vector $\bfx\in\R^n$ is called \emph{singular} if for all $\eps>0$, there exists $N_\eps$ such that for all $N\geq N_\eps$, there exist $\bfq\in\Z^n\setminus\{0\}$ and $p\in\Z$ such that
	\[
	\norm{\bfq\cdot\bfx-p}\leq\eps N^{-n} \text{ and } \norm{\bfq}\leq N.
	\]
	Khintchine introduced this notion in \cite{Khintchine1926} and showed that the set of singular vectors has Lebesgue measure zero. In 2011, Cheung \cite{Che11} proved that the Hausdorff dimension of the set of singular vectors in $\R^2$ is $4/3$, and this was extended in 2016 by Cheung and Chevallier \cite{CC16}, who proved that the Hausdorff dimension of the set of singular vectors in $\R^n$ is $n^2/(n+1)$. Similarly, one can define the set of singular $m\times n$ matrices and study its Hausdorff dimension. In 2017 Kadyrov, Kleinbock, Lindenstrauss and Margulis \cite{KKLM17} gave an upper bound $mn(1-\frac{1}{m+n})$, and this upper bound was shown to be sharp by Das, Fishman, Simmons and Urba\'{n}ski in a recent preprint \cite{das2023variational}. We refer the readers to \cite{das2023variational} and the references therein for a detailed history of the study of singularity. For the weighted approximation, Liao, Shi, Solan and Tamam \cite{LSST20} computed the Hausdorff dimension of weighted singular vectors in $\R^2$. In a recent preprint, Kim and Park \cite{kim2022lower} obtained a lower bound of the Hausdorff dimension of weighted singular vectors in $\R^n$.
	
	In this paper, we study the set of singular vectors on proper affine subspaces of $\R^n$. We will give an upper bound of its Hausdorff dimension in terms of the Diophantine exponent of the parameter of the affine subspace.
	
	Let $d<n$ be a positive integer. For a row vector $\bfs=(s_1,\dots, s_d)\in\R^d$, let $\widetilde{\bfs}=(1,s_1,\dots,s_d)\in\R^{d+1}$. We note that up to a permutation of coordinates, any $d$-dimensional affine subspace of $\R^n$ can be given by $\{(\bfs,\widetilde{\bfs}A)\in\R^n: \bfs\in\R^d\}$ for some $A\in \Mat_{d+1,n-d}(\R)$. 
 
 For $A\in\Mat_{d+1,n-d}(\R)$, the Diophantine exponent $\omega(A)$ of $A$ is defined to be the supremum of $\omega$ such that the inequality
	\begin{equation} \label{eq:definition of omega(A)}
	\norm{A\bfq-\bfp}\leq \norm{\bfq}^{-\omega}
	\end{equation}
	has infinitely many solutions $\bfq\in\Z^{n-d}\setminus\{0\}$ and $\bfp\in\Z^{d+1}$. We note that if there exists $\bfq\in \Z^{n-d}$ and $\bfp\in\Z^{d+1}$, not both zero, such that $A\bfq-\bfp=0$, then $\omega(A)=+\infty$.

	\begin{thm}\label{thm:main_sing}
    For $A\in\Mat_{d+1,n-d}(\R)$, let $\mathrm{Sing}_A$ denote the set of $\bfs\in\R^d$ such that $(\bfs, \widetilde{\bfs}A)$ is singular. We have
		\begin{equation*}
		    \dim_H(\mathrm{Sing}_A) \leq 
            \begin{cases}
                \frac{d^2}{d+1}+\frac{nd(\omega(A)-\frac{n-d}{d+1})}{(1+(d+1)\omega(A)-(n-d))(n+1)},\quad &\text{if }\frac{n-d}{d+1} \leq \omega(A) < n-1, \\
                \frac{d^2}{d+1}+\frac{1}{n+1}\big( \omega(A)-\frac{n-d}{d+1} \big),\quad& \text{if } n-1\leq \omega(A) < n, \\
                d,\quad &\text{if }\omega(A)\geq n.
            \end{cases}
		\end{equation*}
	\end{thm}
	
	\begin{remark}
		By Dirichlet's approximation theorem, $\omega(A)\geq\frac{n-d}{d+1}$ for every $A\in\Mat_{d+1,n-d}(\R)$. If $\omega(A)>\frac{n-d}{d+1}$, then $A$ is called \emph{very well approximable}, and it is well-known that the set of very well approximable matrices has Lebesgue measure zero. Hence for Lebesgue almost every $A\in\Mat_{d+1,n-d}(\R)$, the upper bound we get in \Cref{thm:main_sing} is $\frac{d^2}{d+1}$. We also note that the right hand side of the inequality in the theorem is a continuous function in $\omega(A)$.
	\end{remark}

    \subsubsection*{Geometric formulation}
        Given any $d$-dimensional linear subspace $W$ of $\bbP^n(\R)$, we may define the Diophantine exponent $\omega(W)$ of $W$ to be the supremum of $\omega$ such that the following holds: there exist infinitely many hyperplanes $Q\subset \bbP^n(\R)$ defined over $\Q$ such that $d(W, Q) \leq H(Q)^{-\omega-1}$. Here $d(W,Q)=\sup_{[x]\in W}\inf_{[y]\in Q}\frac{\norm{x\wedge y}}{\norm{x}\norm{y}}$, and $H$ denotes the Weil height on $\Gr(n,n+1)(\Q)\cong \bbP^n(\Q)$ associated with $\mathcal{O}_{\bbP^n}(1)$. More explicitly, $Q$ is defined by a linear equation with $\Q$ coefficients. Let $v_Q\in\Z^{n+1}\setminus\{0\}$ denote the primitive integral vector defining $Q$; that is $Q=v_Q^\perp$. Then we define $H(Q)$ to be $\norm{v_Q}$. One can also write $d(W,Q)=\sup_{[x]\in W}\frac{x\cdot v_Q}{\norm{x}\Vert{v_Q}\Vert}$.
        
        Now for $A\in\Mat_{d+1,n-d}(\R)$, let $\cL_A=\{ (\bfs,\widetilde{\bfs}A) : \bfs\in\R^d \}$ be the associated $d$-dimensional affine subspace. Let $W_A$ be the projective closure of $\cL_A$ in $\bbP^n(\R)$. Then one may check that $\omega(W_A)=\omega(A)$.

 \bigskip
 
	To prove \Cref{thm:main_sing}, we shall use Dani correspondence to reformulate the problem in the language of homogeneous dynamics. Let $X_{n+1}$ denote the space of unimodular $(n+1)$-lattices $X_{n+1}=G/\Gamma$, where $G=\SL_{n+1}(\R)$ and $\Gamma=\SL_{n+1}(\Z)$. Let $g_t=\diag(e^{\frac{n}{n+1}t}, e^{-\frac{1}{n+1}t}, \dots, e^{-\frac{1}{n+1}t})$. For $\bfx\in\R^n$, let
	\[
	u(\bfx)=\begin{pmatrix}
	1 & \bfx \\
	0 & I_n
	\end{pmatrix}
	\]
	By Dani correspondence \cite[Theorem~2.14]{Dan85}, $\bfx$ is singular if and only if $u(\bfx)\Gamma$ is \emph{divergent} in $X_{n+1}$; here we say $x\in X_{n+1}$ is divergent if $g_tx$ leaves any compact set in $X_{n+1}$ for all large $t$. We say that $x\in X_{n+1}$ is \emph{$g_t$-divergent on average} if
	\[
	\lim_{N\to\infty}\frac{1}{N}\abs{\{ l\in\{1,\dots,N\}\colon g_lx\in K \}}=0
	\]
	for every compact set $K$ in $X_{n+1}$. Clearly, $x$ is divergent on average if $x$ is divergent. Hence, \Cref{thm:main_sing} follows immediately from the following theorem.
	
	\begin{thm}\label{thm:main_div_on_avg}
		Let $A\in\Mat_{d+1,n-d}(\R)$. Let
		\begin{equation}\label{eq:set of div_on_avg}
		E_A:=\{ \bfs\in\R^d \colon u(\bfx)\Gamma \text{ is $g_t$-divergent on average for } \bfx=(\bfs, \widetilde{\bfs}A) \}.
		\end{equation}
		Then
        \begin{equation} \label{eq:E_A}
            \dim_H(E_A) \leq 
            \begin{cases}
                \frac{d^2}{d+1}+\frac{nd(\omega(A)-\frac{n-d}{d+1})}{(1+(d+1)\omega(A)-(n-d))(n+1)}, &\text{if }\omega(A) < n-1, \\
                \frac{d^2}{d+1}+\frac{1}{n+1}\big( \omega(A)-\frac{n-d}{d+1} \big),& \text{if } n-1\leq \omega(A) < n, \\
                d, &\text{if } \omega(A)\geq n.
            \end{cases}
        \end{equation}
	\end{thm}

	Our proof of \Cref{thm:main_div_on_avg} follows the approach of \cite{KKLM17} and \cite[\S4]{Kha20}. The novel part is constructing a suitable Margulis height function that satisfies the contraction hypothesis. This Margulis height function will be based on the one constructed in \cite{BQ12} and further explored in \cite{Shi20}. Our construction involves introducing a dynamically defined Diophantine exponent for points in $X_{n+1}$. We explore its relation with the classical Diophantine exponent in \Cref{sect:Diophantine exponents}.

	\subsection*{Acknowledgement}
	We would like to thank Manfred Einsiedler, Alexander Gorodnik, Wooyeon Kim, Dmitry Kleinbock, Nicolas de Saxc\'e, Ronggang Shi, and Shucheng Yu for their helpful discussions.

	\section{Margulis function}\label{sect:Margulis function}
	
	Let $G$ be a connected Lie group, $\Gamma$ a lattice in $G$, and let $X=G/\Gamma$. Let $H\leq G$ be a connected semisimple Lie group without compact factors, and let $Z_G(H)$ denote its centralizer in $G$. Let $\{c_t\}_{t\in\bbR}$ be an $\Ad$-diagonalizable one-parameter subgroup of $H$, and let $\{b_t\}_{t\in\bbR}$ be a one-parameter subgroup of $G$ which is contained in $Z_G(H)$. Let $I$ be the interval $[-\frac{1}{2},\frac{1}{2}]$. Let $U$ be a $c_1$-expanding abelian subgroup of $H$; that is, in each finite-dimensional linear representation of $H$, every $U$-fixed vector has non-negative $c_t$-weight (see \cite{Shi20}). We fix an isomorphism of Lie groups
	\begin{equation*}
	u\colon \bbR^m\to U.
	\end{equation*}
	
	\begin{lem}[\cite{Shi20}, Lemma 4.1]\label{lem:Benoist-Quint height function}
		There exist $\lambda_0>0$ and $T_0>0$ such that for any compact subset $Z$ of $X$ and $t_0\geq T_0$, there exists a lower semicontinuous function $\alpha: X\to [0,\infty]$ and $c>0$ and $b>0$ with the following properties:
		\begin{enumerate}
			\item For every $y\in X$,
			\begin{equation}\label{eq:contraction hypothesis c_t}
			\int_{I^m}\alpha(c_{t_0}u(w)y)\de w 
			\leq ce^{-\lambda_0t_0}\alpha(y)+b.
			\end{equation}
			
			\item $\alpha$ is finite on $HZ$ and bounded on $Z$.
			
			\item $\alpha$ is Lipschitz with respect to the action of $H$, that is, for every compact subset $F$ of $H$ there exists $C\geq1$ such that $\alpha(gy)\leq C\alpha(y)$ for every $y\in X$ and $g\in F$.
			
			\item $\alpha$ is proper, that is, if $\alpha(Z_0)$ is bounded for some subset $Z_0$ of $X$, then $Z_0$ is relatively compact.
		\end{enumerate}
	\end{lem}
	
	If \eqref{eq:contraction hypothesis c_t} holds, we shall say that $\alpha$ satisfies \emph{$\lambda_0$-contraction hypothesis} for $(c_t, U)$. Such a function is referred to as \emph{Margulis height function}. The function $\alpha$ was originally constructed in \cite{BQ12}.
	
	Let $\alpha\colon X\to[0,\infty]$ be a lower semicontinuous function satisfying (1)-(4) in \Cref{lem:Benoist-Quint height function}. Let $0<\delta<\lambda_0$ be a constant. We construct a function $\at=\at_\delta\colon X\to[0,\infty]$ as follows:
	\begin{equation}\label{eq:definition of dynamical height function}
	\at(y)=\at_\delta(y)=\int_{0}^{\infty}e^{-\delta t}\alpha(b_ty)\de t.
	\end{equation}
 
	Now let $g_t=c_tb_t$.
	\begin{lem}\label{lem:dynamical height function}
            Let $\at$ be as in \eqref{eq:definition of dynamical height function}.
		\begin{enumerate}
			\item There exists a constant $\widetilde{b}>0$ such that for every $y\in X$,
			\begin{equation}\label{eq:contraction hypothesis g_t}
			\int_{I^m}\at(g_tu(w)y)\de w 
			\leq ce^{-(\lambda_0-\delta)t}\at(y)+\widetilde{b}.
			\end{equation}
			
			\item $\at$ is Lipschitz with respect to the action of $H$.
			
		\end{enumerate}
	\end{lem}
	
	\begin{proof}
		Let $\widetilde{b}=\delta^{-1}b$. For $t_0\geq T_0$, we have
		\begin{equation*}
		\begin{split}
		\int_{I^m}\at(g_{t_0}u(w)y)\,\de w &=\int_{I^m}\at(c_{t_0}u(w)b_{t_0}y)\,\de w \\
		&=\int_{I^m}\int_{0}^{\infty}e^{-\delta t}\alpha(c_{t_0}u(w)b_{t+t_0}y)\,\de t\,\de w, \quad\text{ by }\eqref{eq:definition of dynamical height function}\\
		&=\int_{0}^{\infty}e^{-\delta t}\int_{I^m}\alpha(c_{t_0}u(w)b_{t+t_0}y)\,\de w\,\de t,\\
		&\leq \int_{0}^{\infty}e^{-\delta t}(ce^{-\lambda_0t_0}\alpha(b_{t+t_0}y)+b)\,\de t,\quad\text{ by }\eqref{eq:contraction hypothesis c_t}\\
		&=ce^{-\lambda_0t_0}\int_{t_0}^\infty e^{\delta t_0}e^{-\delta t'}\alpha(b_{t'}y)\,\de t'+\widetilde{b},\quad t'=t+t_0\\
		&\leq ce^{-(\lambda_0-\delta)t_0}\int_0^\infty e^{-\delta t'}\alpha(b_{t'}y)\,\de t'+\widetilde{b}\\
		&=ce^{-(\lambda_0-\delta)t_0}\at(y)+\widetilde{b}.
		\end{split}
		\end{equation*}
		This verifies \eqref{eq:contraction hypothesis g_t}.

    Since $\alpha$ is Lipschitz with respect to the action of $H$ and $b_t$ commutes with $H$, it follows that $\at$ is Lipschitz with respect to the action of $H$. 
    \end{proof}

For any $y\in G/\Gamma$, we define its \emph{exponent} with respect to $\{b_t\}$ and $\alpha$ to be
	
	\begin{equation}\label{eq:def_exponent}
	\rho(y;b_t,\alpha)=\limsup_{t\to\infty}\frac{\log\alpha(b_ty)}{t}.
	\end{equation}
	If $\{b_t\}$ and $\alpha$ are clear from the context, we will simply write $\rho(y)$ for $\rho(y;b_t,\alpha)$.
	
	\begin{remark}\label{rmk:properties of exponent}
		From the definition of $\tilde \alpha$, it follows that $\rho$ is invariant under the action of $b_t$. And by (2) of \Cref{lem:dynamical height function}, $\rho$ is invariant under the action of $H$. 
	\end{remark}

        \begin{remark}\label{rmk:finiteness of Margulis function}
           If $\int_{0}^{t_1} e^{-\delta t}\alpha(b_ty)\de t$ for all $t_1\geq 0$, and $\rho(y)<\delta$, then 
           \[
           \at(y)=\at_\delta(y)<+\infty.
           \]
           To see this, suppose $\delta_1$ is such that $\rho(y;b_t,\alpha)<\delta_1<\delta$. Then for some $t_1>0$, $\alpha(b_ty)\leq e^{\delta_1t}$ for all $t\geq t_1$. Hence
            \begin{align*}
               \int_{0}^{\infty}e^{-\delta t}\alpha(b_ty)\de t
               \leq \int_{0}^{t_1} e^{-\delta t}\alpha(b_ty)\de t + \int_{t_1}^\infty e^{-\delta t}e^{\delta_1 t}\de t.
            \end{align*}
        \end{remark}
	
	\section{Contraction hypothesis}\label{sect:ContractionHypothesis}
	Let $G=\SL_{n+1}(\bbR)$ and  $H=H_d=\begin{pmatrix}
	\SL_{d+1}(\bbR) & \\
	& I_{n-d}
	\end{pmatrix}$. Consider the one-parameter diagonal subgroups of $G$ defined by the following:
	\begin{align}
	&g_t=\begin{pmatrix}
	e^{\frac{n}{n+1}t} & \\
	& e^{-\frac{1}{n+1}t}I_n
	\end{pmatrix},\label{eq:gt}\\
	&b_t=\begin{pmatrix}
	e^{\frac{n-d}{(d+1)(n+1)}t}I_{d+1} & \\
	& e^{-\frac{1}{n+1}t}I_{n-d}
	\end{pmatrix}\in Z_G(H)\text{, and } \label{eq:bt}\\
	&c_t=\begin{pmatrix}
	e^{\frac{d}{d+1}t} & & \\
	& e^{-\frac{1}{d+1}t}I_d & \\
	& & I_{n-d}
	\end{pmatrix}\in H. \label{eq:ct}
	\end{align}
	Then $g_t=b_tc_t$. We note that this is a specialization of the setting in \Cref{sect:Margulis function}.

        In this particular situation, we would like to give the explicit expression of $\alpha$ in \Cref{lem:Benoist-Quint height function} and compute its contraction rate $\lambda_0$ as in \eqref{eq:contraction hypothesis c_t}.
	
	Restricting the standard action of $G$ on $\bbR^{n+1}$ to $H_d$, we have the following decomposition of $H_d$-modules
	\begin{equation*}
	\bbR^{n+1}=V_0^\perp\bigoplus V_0,
	\end{equation*} 
	where $V_0^\perp$ is the $\bbR$-span of $\{e_0,\dots,e_d\}$ and $V_0$ is the $\bbR$-span of $\{e_{d+1},\dots,e_{n}\}$. Here $V_0^\perp$ is the standard representation of $H_d\cong \SL_{d+1}(\R)$ and $H_d$ acts trivially on $V_0$.
	Taking exterior products, we get
	\begin{equation}\label{eq:decomposition of exterior representations}
	\Lambda^k\bbR^{n+1}=
	\bigoplus_{i=\max\{0,k+d-n\}}^{\min\{d+1,k\}}
	\Lambda^i(V_0^\perp)\bigotimes\Lambda^{k-i}V_0,
	\end{equation}
	because if $\Lambda^i(V_0^\perp)\neq 0$, then $0\leq i\leq d+1$, and if $\Lambda^{k-i}V_0\neq 0$, then $0\leq k-i\leq n-d$.  
	
 Fix $\max\{0,k+d-n\}\leq i\leq \min\{d+1,k\}$ and $w\in\Lambda^i(V_0^\perp)$. For any  $0\leq j_1<j_2<\cdots<j_i\leq d+1$ and $J=\{j_1,\ldots,j_i\}$, we write $e_J=e_{j_1}\wedge e_{j_2}\wedge \cdots \wedge e_{j_i}$. We express $\Lambda^i(V_0^\perp)=V_+\oplus V_-$, where the tensors $e_J$ with $0\in J$ (resp.\ $0\not\in J$), form a basis of $V_+$ (resp.\ $V_-$). Let $\pi_+$ (resp.\ $\pi_-$) be the projection from $\Lambda^i(V_0^\perp)$ to $V_+$ (resp.\ $V_-$). By \eqref{eq:ct}, $c_t$ acts as $e^{\frac{d+1-i}{d+1}t}$ (resp.\ $e^{-\frac{i}{d+1}t}$) on $V_+$ (resp.\ $V_-$).
 
 {\it Suppose that $i\leq d$ and $\pi_-(w)\neq 0$.} For $\bfs\in\R^d$, by abuse of notation, let
	\[
	u(\bfs)=\begin{psmallmatrix}
	1 & \bfs & 0\\
	& I_d & \\
	& & I_{n-d}
	\end{psmallmatrix}.
	\]
	Consider the affine map $f_w\colon\bbR^d\to\bbR^{\binom{d}{i-1}}\cong V_+$ given by
$\bfs\mapsto \pi_+(u(\bfs)w)$.

	\subsection{Dimension of the image of $f_w$}
First, we note that
	\begin{equation*}
	u(\bfs)e_j=\begin{cases}
	e_0 & j=0, \\
	s_je_0+e_j & 1\leq j\leq d.
	\end{cases}
	\end{equation*}
	Then for $0\leq j_1<\cdots<j_i\leq d$ and $J=(j_1,\ldots,j_i)$, we have
	\begin{align*}
	u(\bfs)e_J=
	\begin{cases}
    e_J & j_1 = 0, \\
	\sum_{k=1}^{i}(-1)^{k-1}s_{j_k}e_{\{0\}\cup J\setminus\{j_k\}}+e_J & j_1\geq 1.
	\end{cases}
	\end{align*}
        
	We claim that 
	\begin{equation} \label{eq:lower bound of the affine map}
	\dim(\Ima f_w)\geq i.
	\end{equation}
	
 Let $f(\bfs)=f_w(\bfs)-\pi_+(w)$ for all $\bfs\in\R^d$. Since $u(\bfs)$ acts trivially on $V_+$, $f:\R^d\to V_+$ is a linear map, and $\Ima f_w=\pi_+(w)+\Ima f$. Let $M_f$ denote the matrix of $f$ under the standard basis. Since $\pi_-(w)\neq 0$, there exists an $i$-multi-index $J$ with $0\notin J$ such that $w_J$, the $e_J$-component of $w$, is nonzero. We take the $i\times i$ minor of $M_f$ corresponding to the basis vectors $\{e_j: j\in J\}\subset \R^d$ and $\{e_{\{0\}\cup J\setminus\{j\}} : j\in J \}\subset V_+$. Up to a permutation, this is a diagonal matrix with entries $\pm w_J\neq 0$. Hence the projection of $\Ima f$ on the span of $\{e_{\{0\}\cup J\setminus\{j\}} : j\in J \}$ is surjective. So $\dim(\Ima f)\geq i$. This proves that claim. 
	
	\subsection{Expansion of vectors}
	The following lemma can be viewed as a special case of \cite[Lemma 3.5]{Shi20}, but here we specify explicit exponents. 
 
	\begin{lem}\label{lem:expansion of vectors}
		Let $1\leq i\leq d$ and $V=\Lambda^i\bbR^{d+1}$. For every $0<\theta<i$, there exists $C=C_\theta>0$ such that for every $t > 0$ and every $v\in V\setminus\{0\}$ we have
		\begin{equation}\label{eq:contraction_vector}
		\int_{I^d}\norm{c_tu(\bfs)v}^{-\theta}\de \bfs\leq C_\theta e^{-\frac{d+1-i}{d+1}\theta t}\norm{v}^{-\theta}.
		\end{equation}
	\end{lem}                         
	
	\begin{proof}
		We note that $V_+$ (resp.\ $V_-$) is the eigenspace of $c_1$ in $V$ with eigenvalue $e^{\frac{d+1-i}{d+1}}$ (resp.\ $e^{-\frac{i}{d+1}}$), and $V=V_+\oplus V_-$. 
		
		For every $v\in V$ and $r>0$ we set
		\begin{equation*}
		D^+(v,r)=\{ \bfs\in I^d\colon \norm{\pi_+(u(\bfs)v)}\leq r \}.
		\end{equation*}
		We claim that there exists $C>0$ such that for any unit vector $v$ and $r>0$, 
		\begin{equation}\label{eq:measure of disc}
		\abs{D^+(v,r)}< Cr^i,
		\end{equation}
            where $\abs{\cdot}$ denotes the Lebesgue measure on $\R^d$.

        We first verify the claim for a fixed $v$.
		If $\pi_-(v)=0$, then $v$ is fixed by $u(\bfs)$ for all $\bfs\in I^d$, and thus \eqref{eq:measure of disc} holds. Otherwise, suppose that $\pi_-(v)\neq 0$. By \eqref{eq:lower bound of the affine map}, we know that the image of the affine map $f_v:\bfs\mapsto{\pi_+(u(\bfs)v)}$ has dimension at least $i$. Hence $\abs{D^+(v,r)}\ll r^{\dim(\Ima f_v)}\leq r^i$ for $0<r<1$. Taking sufficiently large $C$, we can ensure that \eqref{eq:measure of disc} also holds for all $r>0$.

        Next, we show that the constant $C$ can be chosen uniformly for all unit vectors $v$ in $V$. Note that there exist $\eps_1,\eps_2>0$ which only depend on $d$, such that the following holds: if the unit vector $v$ satisfies $\norm{\pi_-(v)}<\eps_1$, then $D^+(v,r)=\emptyset$ for all $0<r<\eps_2$. Hence, \eqref{eq:measure of disc} holds for $C=\eps_{2}^{-i}$ for all such $v$ and all $r>0$. On the other hand, suppose the unit vector $v$ satisfies $\norm{\pi_-(v)}\geq\eps_1$. Let $f(\bfs)=f_v(\bfs)-\pi_+(v)=f_{\pi_-(v)}(\bfs)$ for all $\bfs\in\R^d$. Then, $f$ is a linear map. We consider the singular value decomposition of $f$. Let $\lambda_1(v),\dots,\lambda_p(v)$ be the non-zero singular values, where $p=\dim(\Ima f)$. We note that $f$ depends only on $\pi_-(v)$. By the discussion in the previous subsection, we have $p\geq i$. We order the singular values so that $\abs{\lambda_1(v)}\geq\cdots\geq\abs{\lambda_p(v)}>0$, and define $c(v)=\prod_{k=1}^{i}\abs{\lambda_k(v)}$. Note that $c(v)$ is positive and it varies continuously in $v$, and hence it achieves a positive minimum $c$ on the compact set of unit vectors $v\in V$ satisfying $\norm{\pi_-(v)}\geq\eps_1$. We take $C=2\sqrt{d}c^{-1}$, and then \eqref{eq:measure of disc} holds for all $v$ satisfying $\norm{\pi_-(v)}\geq\eps_1$ and all $r>0$.

        Combining the above two cases, we now take $C=\max\{\eps_2^{-i}, 2\sqrt{d}c^{-1} \}$, and \eqref{eq:measure of disc} holds for all $v$ and $r$. Hence the claim is verified.

        Due to this claim, we have verified \cite[Lemma 3.6]{Shi20}, where $\vartheta_0=i$. The deduction of \eqref{eq:contraction_vector} follows from the proof of \cite[Lemma 3.5]{Shi20} using \cite[Lemma 3.6]{Shi20}, as was done earlier in \cite[Lemma 5.1]{EMM98}.
	\end{proof}
	
	\subsection{Construction of $\alpha$}
	We recall the construction of a Margulis height function $\alpha:G/\Gamma\to [0,+\infty]$ from \cite{BQ12}. We shall specify it to our setting and optimize the constants. 
	
	In view of \eqref{eq:decomposition of exterior representations}, for $\max\{0,k+d-n\}\leq i\leq\min\{d+1,k\}$, let $\pi_i$ denote the projection 
	\[
	\pi_i:\Lambda^k\bbR^{n+1}\longrightarrow\Lambda^i(V_0^\perp)\otimes\Lambda^{k-i}V_0.
	\] 
	We also define 
	\[
	\pifix:\Lambda^k\bbR^{n+1}\longrightarrow (\Lambda^0(V_0^\perp)\otimes\Lambda^{k}V_0)\oplus(\Lambda^{d+1}(V_0^\perp)\otimes\Lambda^{k-d-1}V_0)
	\]
	to be the $H_d$-equivariant projection map from $\Lambda^k\bbR^{n+1}$ to the space of $H_d$-fixed vectors. Note that $\pifix=\pi_0\oplus\pi_{d+1}$.
	
	We take $\delta_k=(n+1-k)k$ for $0\leq k\leq n+1$. Let $\eps>0$ and $0<k<n+1$. For every $v\in\Lambda^k\bbR^{n+1}$ we let
	\begin{equation}\label{eq:def_phi_eps}
	\varphi_{\eps}(v)=\begin{cases}
	\min_{1\leq i\leq d}\eps^{\frac{d+1}{d+1-i}\delta_k}\norm{\pi_i(v)}^{-\frac{d+1}{d+1-i}}, & \text{if } \norm{\pifix(v)}<\eps^{\delta_k},\\
	0, & \text{otherwise}.
	\end{cases}
	\end{equation}
	
	\begin{lem}
		Let $0<\theta<\frac{d}{d+1}$ and $\eps>0$. There exists $C>0$ such that the following holds. For any $1\leq k\leq n$, $v\in\Lambda^k\bbR^{n+1}$ and $t>0$ we have
		\begin{equation}\label{eq:integral inequality for vector}
		\int_{I^d}\varphi_\eps^\theta(c_tu(\bfs)v)\de \bfs \leq Ce^{-\theta t}\varphi_\eps^\theta(v).
		\end{equation}
	\end{lem}
	
	\begin{proof}
		For any $1\leq i\leq d$ and $0<\theta<\frac{d}{d+1}$ we have 
		\begin{equation*}
		\frac{d+1}{d+1-i}\theta
		<\frac{d+1}{d+1-i}\cdot\frac{d}{d+1}
		=\frac{d}{d+1-i}
		\leq i.
		\end{equation*}
		Hence by equivariance and \Cref{lem:expansion of vectors}, there exists $C>0$ such that for all $1\leq i\leq d$ we have
		\begin{align*}
		\int_{I^d}\norm{\pi_i(c_tu(\bfs)v)}^{-\frac{d+1}{d+1-i}\theta}\de \bfs
		&=\int_{I^d}\norm{c_tu(\bfs)\pi_i(v)}^{-\frac{d+1}{d+1-i}\theta}\de \bfs \\
		&\leq Ce^{-\theta t}\norm{\pi_i(v)}^{-\frac{d+1}{d+1-i}\theta}.		
		\end{align*}
		On the other hand, since $\pifix$ is the projection to the space of $H_d$-fixed vectors, we have $\pifix(c_tu(\bfs)v)=\pifix(v)$.
		
		Therefore, if $\norm{\pifix(v)}<\eps^{\delta_k}$, then we have
		\begin{align*}
		\int_{I^d}\varphi_\eps^\theta(c_tu(\bfs)v)\de \bfs 
		&=\int_{I^d}\min_{i\neq 0}\eps^{\frac{d+1}{d+1-i}\delta_k\theta}\norm{\pi_i(c_tu(\bfs)v)}^{-\frac{d+1}{d+1-i}\theta}\de \bfs\\
		&\leq \min_{i\neq 0}Ce^{-\theta t}\eps^{\frac{d+1}{d+1-i}\delta_k\theta}\norm{\pi_i(v)}^{-\frac{d+1}{d+1-i}\theta}\\
		&= Ce^{-\theta t}\varphi_\eps^\theta(v),
		\end{align*}
		and this verifies \eqref{eq:integral inequality for vector}.
		If $\norm{\pifix(v)}\geq \eps^{\delta_k}$, then $\varphi_\eps(c_tu(\bfs)v)=\varphi_\eps(v)=0$, and thus \eqref{eq:integral inequality for vector} also holds.
	\end{proof}
	
	Now for $y\in G/\Gamma$ we define
	\begin{equation}\label{eq:alpha_eps^theta}
	\alpha_\eps^{\theta}(y)=\max_v \varphi_\eps^\theta(v)\in [0,\infty],
	\end{equation}
	where $v$ varies over all nonzero $y$-integral decomposable vectors in $\cup_{k=1}^{n}\Lambda^k\bbR^{n+1}$; that is if $y=g\Z^{n+1}$ for some $g\in\SL(n+1,\R)$, then $v=g(v_1\wedge \cdots \wedge v_k)$ for some $1\leq k\leq n$ and linearly independent $v_1,\ldots,v_k\in\Z^{n+1}$.

\begin{remark} \label{rmk:alpha-y}
Let $C$ be a compact subset of $G$ and $Z=C\Gamma/\Gamma$. Let
\[
\eps=\min\{\norm{g(v_1\wedge \cdots \wedge v_k)}^{1/\delta_k}\neq 0: v_1,\ldots,v_k\in \Z^{n+1},\, 1\leq k\leq n,\,g\in C\}.
\]
Then $\eps>0$. If $\eps>1$, reset $\eps=1$. Let $y\in Z$ and $v$ be a nonzero $y$-integral decomposable vector in $\Lambda^k\bbR^{n+1}$ for some $k\in\{1,\ldots,n\}$. Suppose $\varphi_\eps(v)\neq0$. Then $\norm{\pi_0(v)+\pi_{d+1}(v)}<\eps^{\delta_k}\leq \norm{v}$.
By \eqref{eq:decomposition of exterior representations}, $\norm{v}=\max_{0\leq i\leq d+1} \norm{\pi_{i}(v)}$. So, $\norm{\pi_i(v)}=\norm{v}\geq\eps^{\delta_k}$ for some  $1\leq i\leq d$. Hence by \eqref{eq:def_phi_eps},
  \begin{align*}
\varphi_\eps(v)\leq \eps^{\frac{d+1}{d+1-i}\delta_k}\norm{\pi_i(v)}^{-\frac{d+1}{d+1-i}}
\leq \eps^{\frac{d+1}{d}\delta_k}\eps^{-\delta_k(d+1)}<\eps^{-d\delta_k}.
 \end{align*}
So $\alpha_\eps^{\theta}(y)\leq \eps^{-d(n+1)^2/4}$, $\forall  y\in Z$, as  $\sqrt{\delta_k}=\sqrt{(n+1-k)k}\leq (n+1)/2$.
 \end{remark}
	
	\begin{lem}\label{lem:contraction_alpha}
		Let $0<\theta<\frac{d}{d+1}$ and $\eps>0$. For any $y\in G/\Gamma$ and $t>0$ we have
		\begin{equation}\label{eq:contraction hypothesis alphaEpsTheta}
		\int_{I^d}\alpha_\eps^{\theta}(c_tu(\bfs)y)\de \bfs\leq
		Ce^{-\theta t}\alpha_\eps^{\theta}(y)+b.
		\end{equation}
	\end{lem}
	
	\begin{proof}
		See the proof of \cite[Lemma 4.4]{Shi20}.
	\end{proof}
	
	Now let $\alpha=\alpha_\eps^\theta$. It is shown in the proof of \cite[Lemma 4.1]{Shi20} that $\alpha$ satisfies (1)-(4) of \Cref{lem:Benoist-Quint height function} for $\lambda_0=\theta$.
	From $\alpha$, we build $\at$ as in \eqref{eq:definition of dynamical height function}. We will need the following property of $\at$.
	
	\begin{lem}\label{properness of alpha tilde}
		Let $\{y_i\}_{i\in\N}$ be a sequence in $G/\Gamma$ which tends to infinity; that is, $y_i$ leaves any given compact set for all large $i$. Then we have
		$
		\at(y_i)\to\infty.
		$
	\end{lem}

	\begin{proof}
		Since $y_i$ tends to infinity, by Mahler's compactness criterion, there exists $\{v_i\}_{i\in\N}\subset \R^{n+1}$ such that $v_i$ is $y_i$-integral and $\norm{v_i}\to0$ as $i\to\infty$. Since $\min_{0\leq t\leq 1}\varphi_\eps^\theta(b_tv_i)\to\infty$ as $i\to\infty$, from the definition of $\alpha$ we have $\min_{0\leq t\leq 1}\alpha(b_ty_i)\to\infty$ as $i\to\infty$. On the other hand, it follows from the construction of $\at$ that $\at(y)\geq \int_{0}^{1}e^{-\delta t}\alpha(b_ty)\de t \geq e^{-\delta}\min_{0\leq t\leq 1}\alpha(b_ty)$ for every $y\in G/\Gamma$.
		Hence $\at(y_i)\to\infty$ as $i\to\infty$.
	\end{proof}
	
	\section{Diophantine exponents}\label{sect:Diophantine exponents}
	
	In this section, we relate our dynamically defined Diophantine exponent $\rho$ with a matrix's classical Diophantine exponent $\omega$.
	
	\subsection*{Sizes of components}
	Fix $0<\theta<\frac{d}{d+1}$ and a sufficiently small $\eps>0$. Let $\alpha=\alpha_\eps^\theta: G/\Gamma\to[0,+\infty]$ as in \eqref{eq:alpha_eps^theta}. For $A\in\Mat_{d+1,n-d}(\R)$, let $y_A=u_A\Gamma\in G/\Gamma$, where
	\begin{equation} \label{eq:uA}
	u_A=
	\begin{pmatrix}
	I_{d+1} & A \\
	& I_{n-d}
	\end{pmatrix}.
	\end{equation}
	
	Recall that in \eqref{eq:def_exponent} we defined  
	
	\[
	\rho(y)=\rho(y;b_t,\alpha)=\limsup_{t\to\infty}\frac{\log\alpha(b_ty)}{t}.
	\]
	
	Pick any $0\leq\rho_0\leq \rho(y_A)$ such that $0\leq\rho_0<\theta$; in particular, 
 \[
 1-\rho_0\theta^{-1}>0.
 \]
 By the definitions of $\rho$ and $\alpha$, there exist $1\leq k\leq n$, a sequence $\{v_m\}_{m\in\N}\subset\Lambda^k\Z^{n+1}$ consisting of decomposable vectors, and $t_m\to\infty$ such that
	\begin{equation}\label{eq:growth of phi}
	\varphi_\eps^{\theta}(b_{t_m}u_Av_m) \geq e^{\rho_0t_m}.
	\end{equation}
	By definition of $\varphi_\eps$ in \eqref{eq:def_phi_eps}, we rewrite \eqref{eq:growth of phi} as
	\begin{equation}\label{eq:size_fix}
	\norm{\pifix(b_{t_m}u_Av_m)} \ll 1
	\end{equation}
	and
	\begin{equation}\label{eq:size_i}
	\norm{\pi_i(b_{t_m}u_Av_m)} \ll e^{-\frac{d+1-i}{d+1}\theta^{-1}\rho_0t_m}=e^{-\left(1-\frac{i}{d+1}\right)\theta^{-1}\rho_0t_m},
	\end{equation}
	for all $\max\{1,k+d-n\}\leq i\leq\min\{d,k\}$.
	
	\subsection*{Getting exponentially short vectors}
        In the decomposition \eqref{eq:decomposition of exterior representations}, we note that by \eqref{eq:bt}, $\pi_i$ is $b_t$-equivariant, and $b_t$ acts on $\Lambda^i(V_0^\perp)\bigotimes\Lambda^{k-i}V_0$ as scalar multiplication by 
        \begin{equation} \label{eq:b_t eigenvalues}
        e^{\left(i\frac{n-d}{(d+1)(n+1)}-(k-i)\frac{1}{n+1}\right)t}=e^{\left(\frac{i}{d+1}-\frac{k}{n+1}\right)t}.
        \end{equation}

    We discuss the following two cases:
    
    \subsubsection*{Case 1: $k\leq n-d$}
    Let $E_m=\norm{\pi_0(b_{t_m}u_Av_m)}>0$. 
    We will consider the following two subcases.
    
    \subsubsection*{Case 1.1: $E_m\ll e^{-\theta^{-1}\rho_0t_m}$}
    So, by 
    \eqref{eq:size_fix}, \eqref{eq:size_i} and \eqref{eq:b_t eigenvalues}, we have
	\begin{align*}
	&\norm{\pi_i\big(b_{t_m-\theta^{-1}\rho_0t_m}u_Av_m\big)} \ll e^{-\left(1-\frac{k}{n+1}\right)\theta^{-1}\rho_0t_m},
	\end{align*}
	for all $i$ such that $\max\{0,k+d-n\}\leq i\leq\min\{d+1,k\}$. By \eqref{eq:decomposition of exterior representations},
	\[
	\norm{b_{t_m-\theta^{-1}\rho_0t_m}u_Av_m} \ll e^{-\frac{n+1-k}{n+1}\theta^{-1}\rho_0t_m}.
	\]

	Let $t'_m=t_m-\theta^{-1}\rho_0t_m$. We have
	\begin{equation*}
	\norm{b_{t'_m}u_Av_m}\ll e^{-\frac{(n+1-k)\theta^{-1}\rho_0}{(n+1)(1-\theta^{-1}\rho_0)}t'_m}.
	\end{equation*}
	Then by Minkowski's theorem, there exists a non-zero $w_m\in\Z^{n+1}$ belonging to the rank-$k$ lattice associated with $v_m$ such that
	\begin{equation} \label{eq:btuawm-norm}
	\norm{b_{t'_m}u_Aw_m}\ll e^{-\frac{1}{k}\cdot\frac{(n+1-k)\theta^{-1}\rho_0}{(n+1)(1-\theta^{-1}\rho_0)}t'_m}\leq e^{-\frac{(d+1)\theta^{-1}\rho_0}{(n-d)(n+1)(1-\theta^{-1}\rho_0)}t'_m},
	\end{equation}
    the last inequality follows as $k\leq n-d$ in Case~1.

	Write $w_m=\begin{psmallmatrix}
	-\bfp_m\\ \bfq_m
	\end{psmallmatrix}$, where $\bfq_m\in\Z^{n-d}$ and $\bfp_m\in\Z^{d+1}$. By \eqref{eq:uA} and \eqref{eq:bt},
 \begin{equation} \label{eq:btuawm-pq}
 b_{t'_m}u_Aw_m=
 \begin{pmatrix}
   e^{\frac{(n-d)}{(d+1)(n+1)}t'} (A\bfq_m-\bfp_m)\\
   e^{-\frac{1}{n+1}t'}\bfq_m
 \end{pmatrix}.
 \end{equation}
 Therefore by \eqref{eq:btuawm-norm} and \eqref{eq:btuawm-pq}, we get
	\begin{align}
	\norm{A\bfq_m-\bfp_m}&\ll e^{-[\frac{n-d}{(d+1)(n+1)}+\frac{(d+1)\theta^{-1}\rho_0}{(n-d)(n+1)(1-\theta^{-1}\rho_0)}]t'}\text{, and } \label{eq:Aqp-11}\\
	\norm{\bfq_m} &\ll e^{[\frac{1}{n+1}-\frac{(d+1)\theta^{-1}\rho_0}{(n-d)(n+1)(1-\theta^{-1}\rho_0)}]t'}.
    \label{eq:q-11}	
 \end{align} 
 
 Since $1-\rho_0\theta^{-1}>0$, we have that $t'_m=(1-\theta^{-1}\rho_0)t_m\to\infty$. First, suppose that the set $\{\bfq_m\}_{m\in\N}$ is bounded. By passing to a subsequence, we assume that $\bfq_m=\bfq$ is constant for all $m$. By \eqref{eq:Aqp-11}, $\norm{A\bfq-\bfp_m}\to 0$ as $m\to\infty$, so as $\bfp_m\in \bbZ^{d+1}$, we have $A\bfq-\bfp_m=0$ for all large $m$. If $\bfq=0$, then $\bfp_m=0$, which contradicts $w_m\neq 0$. Therefore $\bfq\neq 0$. So by definition \eqref{eq:definition of omega(A)}, we get $\omega(A)=\infty$. 
 
Now we will assume that $\{\bfq_m\}_{m\in\N}$ is unbounded. Therefore by \eqref{eq:q-11},
\[
\frac{1}{n+1}-\frac{(d+1)\theta^{-1}\rho_0}{(n-d)(n+1)(1-\theta^{-1}\rho_0)}>0,
\]
or equivalently, $(n-d)-(n+1)\theta^{-1}\rho_0>0$. 
From \eqref{eq:Aqp-11} and \eqref{eq:q-11} it follows that
	\begin{equation}\label{eq:exponent bound of A,1}
	\norm{A\bfq_m-\bfp_m}\ll \norm{\bfq_m}^{-[\frac{n-d}{d+1}+\frac{\theta^{-1}\rho_0(n+1)}{(n-d)-(n+1)\theta^{-1}\rho_0}]}.
    \end{equation}
	By the definition \eqref{eq:definition of omega(A)} of $\omega(A)$, \eqref{eq:exponent bound of A,1} can be reformulated as
	\begin{align} \label{eq:bound in case 1.1}
 \omega(A) \geq \frac{n-d}{d+1}+\frac{\theta^{-1}\rho_0(n+1)}{(n-d)-(n+1)\theta^{-1}\rho_0},\\ \text{where }
 (n-d)-(n+1)\theta^{-1}\rho_0>0. \notag
	\end{align}
 
    \subsubsection*{Case 1.2: $E_m\gg e^{-\theta^{-1}\rho_0t_m}$}
    Since each $v_m$ is a decomposable vector in $\Lambda^k\Z^{n+1}$, we would like to use Pl{\"u}cker relations to analyze this case. 

    \subsubsection*{Pl\"{u}cker relations:}
	For $v\in\Lambda^k\R^{n+1}$, we can write
	\[
	v=\sum_{i_1<i_2<\cdots<i_k}C_{i_1\cdots i_k}e_{i_1}\wedge\cdots\wedge e_{i_k}.
	\]
	We have that $[v]$ is in the image of the Pl\"ucker embedding $\Gr(k, n+1)(\R)\hookrightarrow\bbP(\Lambda^k\R^{n+1})$ if and only if the coordinates $C_{i_1\cdots i_k}$ satisfy the following \emph{Pl\"ucker relations}: For any two ordered sequences 
	\begin{gather}\label{eq:Plucker relations}
 \cI=(i_1<\dots<i_{k-1}) \text{ and } \cJ=(j_1<\dots<j_{k+1}) \text{, we have }\\
	\sum_{l=1}^{k+1}(-1)^lC_{i_1\cdots i_{k-1}j_l}C_{j_1\cdots\hat{j_l}\cdots j_{k+1}}=0.
	\end{gather}
We note that $[v]$ is in the image of the Pl\"{u}ker embedding if and only if $v$ is a decomposable vector. 
    
    We write
	\[
	b_{t_m}u_Av_m=\sum_{0\leq i_1<i_2<\cdots<i_k\leq n}C_{i_1\cdots i_k}e_{i_1}\wedge\cdots\wedge e_{i_k}.
	\]
	Let $1\leq l\leq d$. Since we are taking the sup-norm, there exist 
 \[
 d+1\leq p_1<p_2<\cdots<p_k\leq n
 \]
such that $\norm{\pi_0(b_{t_m}u_Av_m)} = \abs{C_{p_1\cdots p_k}}$, 
 and there exist 
 \[
 0\leq q_1 <\cdots<q_{l+1}\leq d < d+1 \leq q_{l+2}<\cdots < q_k\leq n
 \]
 such that $\norm{\pi_{l+1}(b_{t_m}u_Av_m)} = \abs{C_{q_1\cdots q_k}}$. By \eqref{eq:Plucker relations} for the two ordered sequences $\cI=(q_2<\cdots<q_k)$ of size $k-1$ and $\cJ=(q_1<p_1<\cdots<p_k)$ of size $k+1$, and the triangle inequality, we have
    \begin{align}\label{eq:recursion}
	\norm{\pi_{l+1}(b_{t_m}u_Av_m)}\cdot\norm{\pi_0(b_{t_m}u_Av_m)}
    &=\abs{C_{q_1\cdots q_k}}\cdot\abs{C_{p_1\cdots p_k}}\\
    &\leq \sum_{l=1}^k \abs{C_{q_2\cdots q_kp_l}\cdot C_{q_1p_1\cdots\hat{p_l}\cdots p_k}}\\
    &\leq k\norm{\pi_l(b_{t_m}u_Av_m)}\cdot\norm{\pi_1(b_{t_m}u_Av_m)}.
	\end{align}

    By \eqref{eq:size_i} we have $\norm{\pi_1(b_{t_m}u_Av_m)}\ll e^{-\frac{d}{d+1}\theta^{-1}\rho_0t_m}$. By applying \eqref{eq:recursion} recursively, we have
    \begin{equation}
        \norm{\pi_l(b_{t_m}u_Av_m)}\ll E_m^{-(l-1)}e^{-\frac{ld}{d+1}\theta^{-1}\rho_0t_m},\quad \forall\, 0\leq l\leq d.
    \end{equation}
 
	Let $t'_m=(1+d\theta^{-1}\rho_0)t_m+(d+1)\log E_m$. Note that $b_t$ acts by $e^{(\frac{i}{d+1}-\frac{k}{n+1})t}$ on $\pi_i(\Lambda^k V)$. Then we have
    \begin{equation}
        \norm{b_{t'_m}u_Av_m} \ll e^{-\frac{kd\theta^{-1}\rho_0}{(n+1)(1+d\theta^{-1}\rho_0)}t'}.
    \end{equation}
    Then by Minkowski's theorem, there exists a non-zero $w_m\in\Z^{n+1}$ belonging to the rank-$k$ lattice associated with $v_m$, such that
	\begin{equation} \label{eq:btuawm-norm12}
	\norm{b_{t'_m}u_Aw_m}\ll e^{-\frac{d\theta^{-1}\rho_0}{(n+1)(1+d\theta^{-1}\rho_0)}t'_m}.
	\end{equation}
    Write $w_m=\begin{psmallmatrix}-\bfp_m\\ \bfq_m\end{psmallmatrix}$, where $\bfq_m\in\Z^{n-d}$ and $\bfp_m\in\Z^{d+1}$. So by \eqref{eq:btuawm-pq} and \eqref{eq:btuawm-norm12}, 
	\begin{align}
	\norm{A\bfq_m-\bfp_m}&\ll e^{-[\frac{n-d}{(d+1)(n+1)}+\frac{d\theta^{-1}\rho_0}{(n+1)(1+d\theta^{-1}\rho_0)}]t'_m} \label{eq:Aqp-12}\text{, and }\\
	\norm{\bfq_m} &\ll e^{[\frac{1}{n+1}-\frac{d\theta^{-1}\rho_0}{(n+1)(1+d\theta^{-1}\rho_0)}]t'_m}.
 \label{eq:q-12}
	\end{align} 
 Since $t'_m\to\infty$, using \eqref{eq:Aqp-12} as argued in Case~1.1, if $\{\bfq_m\}$ is bounded then $\omega(A)=\infty$. Now assume $\norm{\bfq_m}\to\infty$ as $m\to\infty$. By \eqref{eq:Aqp-12} and \eqref{eq:q-12}, we get
	\begin{equation}\label{eq:exponent bound of A1}
	\norm{A\bfq_m-\bfp_m}\ll \norm{\bfq_m}^{-[\frac{n-d}{d+1}+\frac{d(n+1)\theta^{-1}\rho_0}{d+1}]}.
	\end{equation}
	So, by the definition \eqref{eq:definition of omega(A)} of $\omega(A)$, \eqref{eq:exponent bound of A1} can be reformulated as
	\begin{equation} \label{eq:bound in case 1.2}
	\omega(A) \geq \frac{n-d}{d+1}+\frac{d\theta^{-1}\rho_0(n+1)}{d+1}.
	\end{equation}
	
	\subsubsection*{Case 2: $k>n-d$}
	Now $\Lambda^{k}V_0=\{0\}$, and we have
	\[
	\pifix:\Lambda^k\R^{n+1}\to\Lambda^{d+1}(V_0^\perp)\otimes\Lambda^{k-d-1}V_0.
	\]
	It is straightforward to verify that
	\[
	\norm{\pi_i\big(b_{t_m-\theta^{-1}\rho_0t_m}u_Av_m\big)} \ll e^{-\frac{n+1-k}{n+1}\theta^{-1}\rho_0t_m},\quad\forall i.
	\]
	Equivalently,
	\[
	\norm{b_{t_m-\theta^{-1}\rho_0t_m}u_Av_m} \ll e^{-\frac{n+1-k}{n+1}\theta^{-1}\rho_0t_m}.
	\]

	Let $t'_m=t_m-\theta^{-1}\rho_0t_m$. Then
	\begin{equation*}
	\norm{b_{t'_m}u_Av_m}\ll e^{-\frac{(n+1-k)\theta^{-1}\rho_0}{(n+1)(1-\theta^{-1}\rho_0)}t'_m}.
	\end{equation*}
	Then by Minkowski's theorem, there exists a non-zero $w_m\in\Z^{n+1}$ belonging to the rank-$k$ lattice associated with $v_m$, such that
	\begin{equation*}
	\norm{b_{t'_m}u_Aw_m}\ll e^{-\frac{(n+1-k)\theta^{-1}\rho_0}{k(n+1)(1-\theta^{-1}\rho_0)}t'_m}\leq e^{-\frac{\theta^{-1}\rho_0}{n(n+1)(1-\theta^{-1}\rho_0)}t'_m}.
	\end{equation*}
	
	Write $w_m=\begin{psmallmatrix}
	-\bfp_m\\ \bfq_m
	\end{psmallmatrix}$, where $\bfq_m\in\Z^{n-d}$ and $\bfp_m\in\Z^{d+1}$. Then
	\begin{align}
	\norm{A\bfq_m-\bfp_m}&\ll e^{-[\frac{n-d}{(d+1)(n+1)}+\frac{\theta^{-1}\rho_0}{n(n+1)(1-\theta^{-1}\rho_0)}]t'_m} \text{, and } \label{eq:Aqp2}\\
	\norm{\bfq_m} &\ll e^{[\frac{1}{n+1}-\frac{\theta^{-1}\rho_0}{n(n+1)(1-\theta^{-1}\rho_0)}]t'_m}.
 \label{eq:q2}
	\end{align} 
Since $t'_m\to\infty$, using \eqref{eq:Aqp2} and arguing as in Case~1.1 we can show that if $\{\bfq_m\}$ is bounded, then $\omega(A)=\infty$. Now assume that $\norm{\bfq_m}\to\infty$ as $m\to\infty$. 
	From \eqref{eq:Aqp2} and \eqref{eq:q2} it follows that
	\begin{gather}\label{eq:exponent bound of A}
	\norm{A\bfq_m-\bfp_m}\ll \norm{\bfq_m}^{-[\frac{n-d}{d+1}+\frac{\theta^{-1}\rho_0(n+1)}{(d+1)(n-(n+1)\theta^{-1}\rho_0)}]} \text{, and }\\ \frac{1}{n+1}-\frac{\theta^{-1}\rho_0}{n(n+1)(1-\theta^{-1}\rho_0)}>0.
    \end{gather}
	Therefore by the definition of $\omega(A)$, 
	\begin{gather} \label{eq:bound in case 2}
	\omega(A) \geq \frac{n-d}{d+1}+\frac{\theta^{-1}\rho_0(n+1)}{(d+1)(n-(n+1)\theta^{-1}\rho_0)} \text{, and }\\ (d+1)(n-(n+1)\theta^{-1}\rho_0)>0. \notag
	\end{gather}
	
	Combining the discussions in Case~1.1, Case~1.2, and Case~2, we obtain:
 
	\begin{prop}\label{prop:relating Diophantine exponents}
		Given $A\in\Mat_{d+1,n-d}(\R)$, let $y_A=u_A\Gamma\in G/\Gamma$. Suppose that $\omega(A)<n$. Then
		\begin{equation*} 
		\rho(y_A) \leq 
        \begin{cases}
            \frac{n(d+1)(\omega(A)-\frac{n-d}{d+1})}{(1+(d+1)\omega(A)-(n-d))(n+1)}\theta,&\text{if }\frac{n-d}{d+1}\leq \omega(A)< n-1.\\
            \\
            \big(\omega(A)-\frac{n-d}{d+1}\big)\frac{d+1}{d(n+1)}\theta, &\text{if }n-1\leq \omega(A)< n.
        \end{cases}
		\end{equation*}
	\end{prop}

    \begin{proof}
        We take the minimum of the lower bounds of $\omega(A)$ in \eqref{eq:bound in case 1.1}, \eqref{eq:bound in case 1.2}, and \eqref{eq:bound in case 2}. Since
        \[
        (n-d)-(n+1)\theta^{-1}\rho_0 < (d+1)(n-(n+1)\theta^{-1}\rho_0)
        \]
        the lower bound of $\omega(A)$ in \eqref{eq:bound in case 2} is strictly smaller than the lower bound in \eqref{eq:bound in case 1.1}. 
        And the lower bound of $\omega(A)$ in \eqref{eq:bound in case 2} is smaller than or equal to the bound in \eqref{eq:bound in case 1.2} means that 
        \[
        (d+1)/d\leq (d+1)(n-(n+1)\theta^{-1}\rho_0),
        \]
        equivalently $0\leq\rho_0\leq\frac{dn-1}{d(n+1)}\theta$.
Therefore 
        \begin{equation*}
		\omega(A) \geq 
        \begin{cases}
            \frac{n-d}{d+1}+\frac{\theta^{-1}\rho_0(n+1)}{(d+1)(n-n\theta^{-1}\rho_0-\theta^{-1}\rho_0)}, & \text{if } 0\leq\rho_0\leq\frac{dn-1}{d(n+1)}\theta.\\
            \\
            \frac{n-d}{d+1}+\frac{d(n+1)\theta^{-1}\rho_0}{d+1}, &\text{if }\frac{dn-1}{d(n+1)}\theta\leq\rho_0\leq \theta.
        \end{cases}
		\end{equation*}
        Taking its inverse, we get
        \begin{equation*}
		\rho_0 \leq 
        \begin{cases}
            \frac{n(d+1)(\omega(A)-\frac{n-d}{d+1})}{(1+(d+1)\omega(A)-(n-d))(n+1)}\theta, & \text{if }\frac{n-d}{d+1}\leq \omega(A)< n-1.\\
            \\
            \big(\omega(A)-\frac{n-d}{d+1}\big)\frac{d+1}{d(n+1)}\theta, &\text{if } n-1\leq \omega(A)< n.
        \end{cases}
		\end{equation*}
      We note that the right-hand side is strictly less than $\theta$. Since $0\leq\rho_0\leq \rho(y_A)$ with $\rho_0<\theta$ is arbitrary, the conclusion of the proposition follows.
      \end{proof}
	
	\section{Upper bound of the Hausdorff dimension}
	In this section, we prove \Cref{thm:main_div_on_avg}. We follow the same line of arguments as in \cite{KKLM17} and \cite[Sect.4]{Kha20}.
	
	Given $A\in\Mat_{d+1,n-d}(\R)$, we write $A$ in the block form
	\[
	A=\begin{psmallmatrix}
	A_1 \\
	A_2
	\end{psmallmatrix},
	\]
	where $A_1\in\operatorname{Mat}_{1,n-d}(\R)$ and $A_2\in\operatorname{Mat}_{d,n-d}(\R)$. Let $\bfx=(\bfs, \widetilde{\bfs}A)$. Let
	\begin{equation} \label{eq:z_A}
	z_A=\begin{psmallmatrix}
	1 & & \\
	& I_d & -A_2 \\
	& & I_{n-d}
	\end{psmallmatrix}\in Z_G(\{g_t\}).
	\end{equation}
	Observe that for every $t\geq 0$ and $\bfs\in\R^d$, we have
	\[
	u(\bfx) = z_Au(\bfs)u_A.
	\]
	Recall that $y_A=u_A\Gamma$, and thus 
	\begin{equation}\label{eq:basic identity}
	u(\bfx)\Gamma=z_Au(\bfs)y_A.
	\end{equation}
	
	Let $\eps_1>0$. Let $\theta=d/(d+1)-\eps_1$. In this section, 
 \[
 \text{\em we assume that $\rho(y_A)<\theta$}.
 \]
 Let $\delta=\rho(y_A)+\eps_2$ and we choose $\eps_2>0$ small enough such that $\delta<\theta$. Let $\alpha=\alpha_\eps^\theta$ as in \eqref{eq:alpha_eps^theta} and $\at=\at_\delta$ as in \eqref{eq:definition of dynamical height function}. By \Cref{rmk:alpha-y}, we choose $\eps>0$ such that $\alpha(y_A)=\alpha_\eps^\theta(y_A)<+\infty$. 

    \begin{lem}\label{lem:finiteness along b_t-trajectory}
   For every $T>0$,  $\sup_{t\in[0,T]}\alpha(b_ty)<\infty$.
    \end{lem}
    
    \begin{proof}
    Suppose that $\alpha(b_{t_j}y_A)\to \infty$ for a sequence $t_j\to t\in [0,T]$. By the definition \eqref{eq:alpha_eps^theta} of $\alpha$, after passing to a subsequence, we can pick $k\in\{1,\ldots,d\}$ and non-zero $y_A$-integral decomposable vectors $v_j\in\Lambda^k\R^{n+1}$ such that $\varphi_\eps(b_{t_j}v_j)\to\infty$ as $j\to\infty$.
    By the definition \eqref{eq:def_phi_eps} of $\varphi_\eps$, we have $\norm{\pifix(b_{t_j}v_j)}<\eps^{\delta_k}$, and for all $1\leq i\leq d$, $\pi_i(b_{t_j}v_j)\to 0$ as $j\to\infty$. 
    Note that $\norm{\pifix(b_{t_j}v_j)}=\max\{\norm{\pi_0(b_{t_j}v_j)},\norm{\pi_{d+1}(b_{t_j}v_j)}\}$.  
    Since $b_{t_j}\to b_t$, we conclude that $\pi_i(v_j)\to 0$ for all $1\leq i\leq d$, and $\sup_{j}\norm{\pifix(v_j)}<\infty$. Since $v_j=\sum_{i=0}^{d+1} \pi_i(v_j)$, the sequence $\{v_j\}$ is bounded. Now $v_j$ being $y_A$-integral for each $j$, the sequence $\{v_j\}$ is discrete.  Therefore, after passing to a subsequence, we may assume that $v_j=v\neq 0$ for all $j$. Hence $\pi_i(b_tv)=0$ for all $1\leq i\leq d$ and $\norm{\pi_0(b_tv)}<\eps^{\delta_k}$ and $\norm{\pi_{d+1}(b_tv)}=\eps^{\delta_k}$. Since $v$ is a decomposable vector in $\Lambda^k\R^{n+1}$,  by \eqref{eq:recursion} for $l=d$, $\norm{\pi_0(b_tv)}\cdot\norm{\pi_{d+1}(b_tv)}=0$. So, $\pi_0(b_tv)=0$ or $\pi_{d+1}(b_tv)=0$.

 We recall that for any $t'\geq 0$, $\pi_i$ is $b_{t'}$-equivariant, and by \eqref{eq:b_t eigenvalues}, $b_{t'}$ acts by a scaler on $\Ima(\pi_i)$ for each $i$; on $\Ima(\pi_{d+1})$ it acts as an expansion by $e^{(1-k/(n+1))t'}$, and on $\Ima(\pi_0)$ it acts as a contraction by $e^{-(k/(n+1))t'}$. 
    
    Suppose $\pi_0(b_tv)=0$. Then $\norm{\pi_{d+1}(v)}\leq\norm{\pi_{d+1}(b_tv)}<\eps^{\delta_k}$. We also have $\pi_i(b_tv)=0=\pi_i(v)$ for all $1\leq i\leq d$. Hence $\varphi_\eps(v)=+\infty$, and so $\alpha(y_A)=\alpha_{\eps}^\theta(y_A)=+\infty$, which contradicts our choice of $\eps$. 

    Suppose $\pi_{d+1}(b_tv)=0$, then $\norm{\pi_{0}(b_tv)}<\eps^{\delta_k}$. So $\norm{\pi_{0}(b_{t'}v)}<\eps^{\delta_k}$ for all $t'\geq t$. We also have $\pi_i(b_{t}v)=0=\pi_i(b_{t'}v)$ for all $1\leq i\leq d$. So $\varphi_\eps(b_{t'}v)=+\infty$, and hence $\alpha(b_{t'}y_A)=+\infty$ for all $t'\geq t$. So by \eqref{eq:def_exponent}, $\rho(y_A)=+\infty$, contradicting our assumption that $\rho(y_A)<\theta$. 
\end{proof}

    Now by \Cref{rmk:finiteness of Margulis function} and \Cref{lem:finiteness along b_t-trajectory},  $\at(g_tu(\bfs)y_A)<+\infty$ for every $t\geq 0$ and $\bfs\in\R^d$. By \Cref{lem:contraction_alpha}, $\alpha$ satisfies the $\theta$-contraction hypothesis for $c_t$. So, by \Cref{lem:dynamical height function}, $\at$ satisfies the $(\theta-\delta)$-contraction hypothesis for $g_t$. 
	
	Recall that $I=\big[-\frac{1}{2},\frac{1}{2}\big]$. For $x\in X$, $M,t>0$ and $m,N\in\N$, let
	\[
	\begin{split}
	&B_x(M,t,m;N)\\
 &=\{ \bfs\in I^d : \at(g_{mt}u(\bfs)y_A)<M,\, \at(g_{(m+l)t}u(\bfs)y_A)\geq M,
	\text{ for } 1\leq l\leq N \}.
	\end{split}
	\]

    The following shadowing lemma relates random walks with flow trajectories.

    \begin{lem}{\rm (\cite[Lemma~4.8]{Shi20})}\label{lem:shadowing}
        For $1\leq i \leq d$, let $J_i\subset I$ be an interval of length $\abs{J_i}\geq e^{-Nt}$, and $J=\prod_{i=1}^{d}J_i$. Then for any nonnegative measurable function $\psi$ on $X$ one has
        \begin{equation*}
            \int_J\psi(g_{(N+1)t}u(s)x)\de s \leq \int_J\left( \int_{I^d}\psi(g_tu(w)g_{Nt}u(s)x)\de w \right) \de s.
        \end{equation*}
    \end{lem}
 
	\begin{prop}\label{prop:long excursion}
		There exists $c_0\geq 1$ such that for every $t>0$, there exists $M_0=M_0(t)>0$ such that for all $M>M_0$, all $x\in X\setminus\{\at(x)=\infty\}$, and all positive integers $m,N$, one has
		\[
		\abs{B_x(M,t,m;N)\cap B_0} \leq c_0^Ne^{-(\theta-\delta)Nt}\abs{B_0}.
		\]
        for any ball $B_0$ of radius $e^{-mt}$ in $I^d$.
	\end{prop}	

	\begin{proof}
		See the proof of \cite[Propopsition~4.6]{Kha20}. The only change needed is that we replace \cite[Lemma~4.5]{Kha20} with \Cref{lem:shadowing}. 
	\end{proof}

	\begin{cor}\label{cor:long excursion}
		There exists $C_2\geq 1$ such that the following holds. Suppose $M_0$ and $c_0$ are as in \Cref{prop:long excursion}, then for all $M>C_2M_0$, $t>0$ and positive integers $m,N\in\N$, the set $B_x(M,t,m;N)\cap B_0$ can be covered with $c_1^Ne^{(d-(\theta-\delta))Nt}$ balls of radius $e^{-(m+N)t}$, where $B_0$ is any ball of radius $e^{-mt}$ and $c_1=C_2c_0$.
	\end{cor}

        \begin{proof}
            The deduction of the corollary from \Cref{prop:long excursion} is the same as the deduction of \cite[Corollary~5.2]{KKLM17} from \cite[Proposition~5.1]{KKLM17}, or  the deduction of \cite[Corollary~4.7]{Kha20} from \cite[Proposition~4.6]{Kha20}. 
        \end{proof}

	For any $M>0$, let $X_{\leq M}=\{y\in X:\at(y)\leq M\}$. Define
	\[
	Z(M,N,t,\eta) = \{ \bfs\in I^d : \#\{ 1\leq l\leq N:g_{lt}u(\bfs)y_A\notin X_{\leq M} \} > \eta N \}.
	\]
	
 We will prove our main theorem using the following covering result. 
 
	\begin{prop}\label{prop:covering}
		There exists $C_3\geq 1$ such that for all $t>0$, there exists $M_1=M_1(t)>0$ such that for all $M>M_1$, $\eta>0$, and $N\in\N$, the set $Z(M,N,t,\eta)$ can be covered by at most $C_3^Ne^{(d-\eta(\theta-\delta))Nt}$ balls of radius $e^{-Nt}$.
	\end{prop}

 \begin{proof} As in \cite[Proposition~4.8]{Kha20} or \cite[Theorem~1.5]{KKLM17}, the result is a consequence of \Cref{cor:long excursion}. 
 \end{proof}

\begin{remark} [Lipschitz property] By \Cref{lem:dynamical height function}(2), the Margulis height function $\at$ used in this article is Lipschitz with respect to the action of $H$, but not with respect to the action of $\{b_t\}$. We note that in both \cite{KKLM17} and \cite{Kha20}, the proofs of the results quoted above to obtain \Cref{prop:long excursion}, \Cref{cor:long excursion}, and \Cref{prop:covering}, the Lipschitz property of the Margulis height function was used only with respect to the action of the unipotent subgroup $U=\{u(\bfs):\bfs\in\R^d\}\subset H$, but not with respect to the action of $\{g_t\}$.
\end{remark}
	
	\begin{proof}[Proof of \Cref{thm:main_div_on_avg}]
	As in the proof of \cite[Theorem 1.1]{KKLM17} or \cite[Theorem 4.3]{Kha20}, \Cref{prop:covering} implies that the Hausdorff dimension of the set 
		\begin{equation}\label{eq:set of div_on_avg1}
		\{ \bfs\in\R^d \colon u(\bfs)y_A \text{ is $g_t$-divergent on average}\}
		\end{equation}
		is at most $d-(\theta-\delta) = \frac{d^2}{d+1}+\rho(y_A)+\eps_1+\eps_2$. In view of \eqref{eq:z_A}, since $z_A$ commutes with $g_t$, $x\in X_{n+1}$ is $g_t$-divergent on average if and only if $z_Ax$ is so. Hence, in view of \eqref{eq:basic identity}, the set in \eqref{eq:set of div_on_avg1} coincides with the set $E_A$ defined in \eqref{eq:set of div_on_avg}. Finally, since $\eps_1>0$ is arbitrary, \eqref{eq:E_A} follows from \Cref{prop:relating Diophantine exponents}. 
	\end{proof}

\end{document}